\documentclass[twoside,a4paper,reqno]{amsart}
\usepackage[latin1]{inputenc} 
\usepackage{amsmath}
\usepackage{amsfonts} 
\usepackage{amsthm}
\usepackage{amscd}
\usepackage{amsopn}
\usepackage{amstext}
\usepackage{amsxtra}
\usepackage{amssymb}
\usepackage{stmaryrd}
\usepackage{graphicx}
\usepackage{textcase} 
\usepackage{clrscode}
\renewcommand{\Comment}{$\hspace*{-0.075em} //$ } 
\usepackage{bm} 
\numberwithin{equation}{section}
\numberwithin{figure}{section}

\newcounter{algorithm}

\theoremstyle{plain}
\newtheorem{thm}{Theorem}[section]
\newtheorem{lem}[thm]{Lemma}
\newtheorem{cl}[thm]{Corollary}
\newtheorem{pr}[thm]{Proposition}

\theoremstyle{definition}

\theoremstyle{remark}

\providecommand{\abs}[1]{\left\lvert #1 \right\rvert}

\providecommand{\set}[1]{\left\lbrace #1 \right\rbrace}
\providecommand{\gen}[1]{\left\langle #1 \right\rangle}

\renewcommand{\Pr}[1]{\operatorname{Pr}[ #1 ]}

\newcommand{\field}[1]{\mathbb{#1}}

\newcommand{\N}{\field{N}}
\newcommand{\Z}{\field{Z}}

\newcommand{\OO}{\field{O}}

\newcommand{\F}{\field{F}}
\newcommand{\PS}{\field{P}}

\newcommand{\MAGMA}{\textsc{Magma}}

\newcommand{\OV}{\mathcal{O}}

\DeclareMathOperator{\GL}{GL}

\DeclareMathOperator{\IM}{Im}

\DeclareMathOperator{\SL}{SL}

\DeclareMathOperator{\SO}{SO}
\DeclareMathOperator{\Sz}{Sz}

\DeclareMathOperator{\chr}{char}

\DeclareMathOperator{\PSL}{PSL}
\DeclareMathOperator{\diag}{diag}
\DeclareMathOperator{\Tr}{Tr}
\DeclareMathOperator{\SLP}{\mathtt{SLP}}
\DeclareMathOperator{\G2}{{^2}G_2}

\DeclareMathOperator{\Ree}{Ree}

\DeclareMathOperator{\Norm}{N}
\DeclareMathOperator{\Cent}{C}
\DeclareMathOperator{\Zent}{Z}

\DeclareMathOperator{\Dih}{D}
\DeclareMathOperator{\Ker}{Ker}

\DeclareMathOperator{\PPSL}{PSL}

\DeclareMathOperator{\antidiag}{antidiag}

\newcommand{\OR}[1]{\operatorname{O} ( #1 )}

\title{Recognising the small Ree groups in their natural representations}
\author{Henrik B\"a\"arnhielm}
\address{Department of Mathematics \\ University of Auckland \\ New Zealand}
\urladdr{http://www.math.auckland.ac.nz/\textasciitilde henrik/}
\email{henrik@math.auckland.ac.nz}
\keywords{matrix group recognition, exceptional groups, constructive recognition}

\begin{document}

\begin{abstract}
  We present Las Vegas algorithms for constructive recognition and
  constructive membership testing of the Ree groups ${^2}\mathrm{G}_2(q) =
  \Ree(q)$, where $q = 3^{2m + 1}$ for some $m > 0$, in their natural
  representations of degree $7$. The input is a generating set $X \subset \GL(7, q)$.
  
  The constructive recognition algorithm is polynomial
  time given a discrete logarithm oracle. The constructive membership
  testing consists of a pre-processing step, that only needs to be
  executed once for a given $X$, and a main step. The latter is polynomial
  time, and the former is polynomial time given a discrete logarithm
  oracle.
  
  Implementations of the algorithms are available for the computer algebra
  system $\MAGMA$.
\end{abstract}

\maketitle

\section{Introduction}
\label{section:intro}

This paper will consider algorithmic problems for a class of finite
simple groups, as matrix groups over finite fields, given by sets of
generators. The most important problems under consideration are the
following:
\begin{enumerate}
\item The \emph{constructive membership} problem. Given $G = \gen{X} \leqslant \GL(d, q)$ and $g \in \GL(d, q)$, decide whether or not $g \in G$, and if so express
  $g$ as a straight line program in $X$.
\item The \emph{constructive recognition} problem. Given $G = \gen{X} \leqslant \GL(d, q)$, construct an \emph{effective isomorphism} from $G$ to a
\emph{standard copy} $H$ of $G$, together with an effective inverse isomorphism. An isomorphism $\psi : G \to H$ is
effective if $\psi(g)$ can be computed efficiently for every $g \in
G$. 
\end{enumerate}

In \cite{baarnhielm05} we considered these problems for the
Suzuki groups. Here we 
consider the Ree groups ${^2}\mathrm{G}_2(q) = \Ree(q)$, $q = 3^{2m + 1}$ for
any $m > 0$. We only consider the \emph{natural representations},
which have dimension $7$. Our standard copy is
$\Ree(q)$, defined in Section \ref{section:ree_theory}. 

The primary motivation for considering these problems comes from the
\emph{matrix group recognition project} \cite{comp_tree, crlg01, Obrien11}.

The ideas used here for the constructive recognition and
membership testing of $\Ree(q)$ are similar to those used in
\cite{baarnhielm05} and \cite{psl_recognition} for $\Sz(q)$ and
$\SL(2, q)$, respectively. The results are also similar in the sense
that we reduce these problems to the discrete logarithm problem. 


In Section
\ref{section:ree_constructive_membership} we solve the constructive
membership problem for $\Ree(q)$. In Section
\ref{section:ree_conjugacy} we solve the constructive recognition problem for $\Ree(q)$ in the natural representations.

The main objective of this paper is to prove the following:

\begin{thm} \label{main_theorem}
Let $q = 3^{2m+1}$ for some $m > 0$. Assume an oracle for the discrete logarithm
  problem in $\F_q$, with time complexity $\OR{\chi_D}$ field operations, and a random element oracle for subgroups of $\GL(7, q)$, with time complexity $\OR{\xi}$ field operations.
\begin{enumerate}
\item There exists a Las Vegas algorithm that for each $\gen{X}
  \leqslant \GL(7, q)$, such
  that $\gen{X} \cong \Ree(q)$, constructs an effective isomorphism $\Psi :
  \gen{X} \to \Ree(q)$, such that $\Psi^{-1}$ is also effective. The algorithm has expected time complexity 
$\OR{\xi \log\log(q) + \log(q)^2 + \chi_D}$ field operations.
\item There exists a Las Vegas algorithm that for each
  $\gen{X} \leqslant \GL(7, q)$, 
  such that $\gen{X} \cong \Ree(q)$, solves the
  constructive membership problem for $\gen{X}$. The algorithm has
  expected time complexity $\OR{\xi + \log(q)^3}$ field operations and also has a
  pre-processing step, which only needs to be executed once for a given
  $X$, with expected time complexity 
$\OR{(\xi \log\log(q) + \log(q)^3 + \chi_D) \log{\log (q)}^2}$
  field operations. The length of the returned $\SLP$ is $\OR{(\log (q) \log \log (q))^2}$.
\end{enumerate}
\end{thm}

Implementations of the algorithms have been done in $\MAGMA$ \cite{magma}. 

A version of the material in this paper appeared in
\cite{baarnhielm_phd}, relying on a few conjectures. Advice by Bill Kantor and Gunter Malle has led to proofs of the conjectures, for which we are very grateful. In particular, the central idea behind the algorithm in Section
\ref{section:ree_conjugacy} is due to Bill Kantor.

We also thank John Bray, Peter Brooksbank, Alexander
Hulpke, Charles Leedham-Green, Eamonn O'Brien, Maud de Visscher, Robert Wilson and the anonymous referee for their helpful comments.

\section{Preliminaries}
We will now briefly discuss some general concepts that are needed later.

\subsection{Complexity} \label{complexity_general} 
Time complexity is measured in field operations. Basic
matrix arithmetic in $\GL(7, q)$ requires $\OR{1}$ field operations. Raising a matrix to
an $\OR{q}$ power requires $\OR{\log q}$ field operations, for example
using \cite[Lemma 10.1]{classical_recognise}.

We never need to compute large precise orders of matrices. 
It is sufficient to compute \emph{pseudo-orders} \cite[Section $8$]{babaibeals01}.
This can be done using \cite{crlg95}, in $\OR{\log(q) \log\log(q)}$ field
operations.

We shall assume an oracle for the discrete logarithm problem in
$\F_q$ \cite[Chapter $3$]{MR1745660}, requiring $\OR{\chi_D}$ field operations.

\subsection{Straight line programs} \label{section_slp}

For constructive membership testing, we want to express an element of
a group $\gen{X}$ as a \emph{straight line program} in $X$,
abbreviated to $\SLP$. An $\SLP$ is a data structure for a word, which
allows for efficient computations \cite[Section 1.2.3]{seress03}.

\subsection{Random group elements}
\label{section:random_elements}

Our algorithms need to construct (nearly) uniformly distributed random
elements of a subgroup of $\GL(7, q)$. The algorithm
of \cite{babai91} solves this task in polynomial time, but it is not commonly used in practice. The \emph{product replacement} algorithm of
\cite{lg95} also solves this task. It is fast in practice and polynomial time
\cite{Pak00}.

We shall assume that we have a random element oracle, which produces a
uniformly random element of $\gen{X} \leqslant \GL(7, q)$ using $\OR{\xi}$ field
operations, and returns it as an $\SLP$ in $X$.

An important issue is the length of the $\SLP$s that are computed. The
length of the $\SLP$s must be polynomial, otherwise evaluation would not be
polynomial time. We assume that $\SLP$s of random
elements have length $\OR{n}$ where $n$ is the number of random
elements that have been selected so far during the execution of the
algorithm. 

In \cite{lgsm02}, a variant of the product replacement algorithm is
presented that constructs random elements of the normal closure of a
subgroup. This will be used here to construct random elements of the
derived subgroup of a group $\gen{X}$, using the fact that this is
precisely the normal closure of $\gen{[x, y] : x, y \in X}$.

\subsection{Probabilistic algorithms}
\label{section:prob_algorithms}
The algorithms we consider are probabilistic of the type known as \emph{Las Vegas} algorithms. This type of
algorithm is discussed in \cite[Section 3.2.1]{hcgt}. We present Las Vegas algorithms in the same way as in \cite{baarnhielm05}.

\subsection{Recognition of $\PSL(2, q)$}
\label{section:psl_recognition}
In \cite{psl_recognition}, an algorithm for constructive recognition
and constructive membership testing of $\PSL(2, q)$ is presented. 

We will use \cite{psl_recognition} since $\PSL(2, q)$ arises as a
subgroup of $\Ree(q)$. Because
of this, we state the main result here. 

\begin{thm} \label{thm_psl_recognition} Assume an oracle for the
  discrete logarithm problem in $\F_q$. There exists a Las Vegas
  algorithm that, given $\gen{X} \leqslant \GL(d, q)$, which acts
  absolutely irreducibly and cannot be written over a smaller field,
  with $\gen{X} \cong \PPSL(2, q)$ and $q = p^e$, constructs an effective
  isomorphism $\psi : \gen{X} \to \PPSL(2, q)$ and performs
  pre-processing for constructive membership testing. The algorithm has
  expected time complexity
\[ \OR{(\xi + d^3 \log(q) \log \log(q^d)) \log \log(q) + d^7 \abs{X} + d \chi_D + d\xi} \]
 field operations.

 The inverse of $\psi$ is also effective. Each image of $\psi$ can be computed using $\OR{d^3}$ field
 operations, and each pre-image using $\OR{d^3 \log(q)\log\log(q) + e^3}$ field operations.
 After the algorithm has executed, constructive membership testing of $g
 \in \GL(d, q)$ requires $\OR{d^3 \log(q)\log\log(q) + e^3}$ field operations, and the
 resulting $\SLP$ has length $\OR{\log (q) \log\log(q)}$.
\end{thm}

\subsection{Notation}
Some notation will be fixed throughout the paper.

\begin{itemize}
\item For a group $G$ and a prime $p$, let $\mathrm{O}_p(G)$ denote the
largest normal $p$-subgroup of $G$. 
\item If a group $G$ acts on a set $\OV$ and $P
\in \OV$, then $G_P$ denotes the stabiliser in $G$ of $P$.
\item Let $q = 3^{2m+1}$, where $m > 0$, be the size of the finite field $\F_q$. Let $t = 3^m = \sqrt{q/3}$ and let $\omega$ be a fixed primitive element of $\F_q$.
\item Let
\[ \antidiag(x_1, \dotsc, x_7) = \begin{bmatrix}
0 & 0 & 0 & 0 & 0 & 0 & x_1 \\
0 & 0 & 0 & 0 & 0 & x_2 & 0 \\
0 & 0 & 0 & 0 & x_3 & 0 & 0 \\
0 & 0 & 0 & x_4 & 0 & 0 & 0 \\
0 & 0 & x_5 & 0 & 0 & 0 & 0 \\
0 & x_6 & 0 & 0 & 0 & 0 & 0 \\
x_7 & 0 & 0 & 0 & 0 & 0 & 0 
\end{bmatrix} \]

\item For a module $M$ or a matrix $g$, we denote the symmetric square of $M$ or $g$ by $\mathcal{S}^2(M)$ and $\mathcal{S}^2(g)$, respectively.
\item We will denote our standard copy of the small Ree group, defined in Section \ref{section:ree_theory}, by $\mathfrak{S}$ and $\Ree(q)$.
\item We will denote the natural module of $\PSL(2, q)$ by $\mathfrak{A}$.
\item The time complexity in field operations for an invocation of a random element oracle on a group $G \leqslant \GL(7, q)$ will be denoted $\xi$. 
\item The time complexity in field operations for an invocation of a discrete logarithm oracle on $\F_q$ will be denoted $\chi_D$. 
\item For $n \in \N$, let $\sigma_0(n)$ be the number of divisors of $n$. Note that $\sigma_0(n) \leqslant n$ and from \cite[pp. $64, 359,
262$]{MR568909}, for every $\varepsilon > 0$, if $n$ is
sufficiently large, then $\sigma_0(n) < 2^{(1 + \varepsilon)
  \log_{\mathrm{e}}(n) / \log\log_{\mathrm{e}}(n)}$. 

\item For a vector space $V$, we denote the corresponding projective space by $\PS(V)$.
\item We denote the standard $n$-dimensional vector space over $\F_q$ by $\F_q^n$, and the corresponding projective space by $\PS^{n-1}(\F_q)$.
\item We denote the dihedral group of order $n$ by $\Dih_n$.
\item We denote the Euler totient function by $\phi$.
\item We denote the Frobenius automorphism by $\varphi$.
\end{itemize}

\section{The small Ree groups} \label{section:ree_theory}

The small Ree groups were first described in \cite{MR0125154, small_ree}. An elementary construction is given in \cite[Chapter $4.5$]{MR2562037}.


 We now define our standard copy of the Ree groups. The generators we use are those described in \cite{ree_gens}. For $x \in \F_q$ and
 $\lambda \in \F_q^{\times}$, define the matrices
\begin{equation}
\alpha(x) = \begin{bmatrix}
1 & x^t & 0 & 0 & -x^{3 t + 1}  & -x^{3t + 2}  & x^{4t + 2}  \\
0 & 1 & x &  x^{t + 1} & -x^{2t + 1}  & 0 & -x^{3t + 2} \\
0 & 0 & 1 & x^t & -x^{2 t} & 0 & x^{3t + 1}  \\
0 & 0 & 0 & 1 & x^t & 0 & 0 \\
0 & 0 & 0 & 0 &	1 & -x & x^{t + 1} \\
0 & 0 & 0 & 0 & 0 & 1 & -x^t \\
0 & 0 & 0 & 0 & 0 & 0 &	1
\end{bmatrix} 
\end{equation}
\begin{equation}
\beta(x) = \begin{bmatrix}
1 & 0 & -x^t & 0 & -x & 0 & -x^{t + 1} \\
0 & 1 & 0 & x^t & 0 & -x^{2t} & 0 \\
0 & 0 & 1 & 0 & 0 & 0 & x \\
0 & 0 & 0 & 1 & 0 & x^t & 0 \\
0 & 0 & 0 & 0 &	1 & 0 & x^t \\
0 & 0 & 0 & 0 & 0 & 1 & 0 \\
0 & 0 & 0 & 0 & 0 & 0 &	1
\end{bmatrix} 
\end{equation}
\begin{equation}
\gamma(x) = \begin{bmatrix}
1 & 0 & 0 &  -x^t & 0 & -x &  -x^{2t} \\
0 & 1 & 0 & 0 &  -x^t & 0 & x  \\
0 & 0 & 1 & 0 & 0 & x^t & 0 \\
0 & 0 & 0 & 1 & 0 & 0 & -x^t \\
0 & 0 & 0 & 0 &	1 & 0 & 0 \\
0 & 0 & 0 & 0 & 0 & 1 & 0 \\
0 & 0 & 0 & 0 & 0 & 0 &	1
\end{bmatrix} 
\end{equation}
\begin{equation}
h(\lambda) = \diag(\lambda^t, \lambda^{1 - t}, \lambda^{2t-1}, 1, \lambda^{1-2t}, \lambda^{t-1}, \lambda^{-t})
\end{equation}
\begin{equation}
\Upsilon = \antidiag(-1, -1, -1, -1, -1, -1, -1) 
\end{equation}

and define the Ree group as
\begin{equation}
\mathfrak{S} = \Ree(q) = \gen{ \alpha(x), \beta(x), \gamma(x), h(\lambda), \Upsilon \mid x \in \F_q, \lambda \in \F_q^{\times}}.
\end{equation}
Also, define the subgroups of upper triangular and diagonal matrices:
\begin{align}
U(q) &= \gen{\alpha(x), \beta(x), \gamma(x) \mid x \in \F_q} \\
H(q) &= \set{h(\lambda) \mid \lambda \in \F_q^{\times}} \cong \F_q^{\times}.
\end{align}

From \cite{ree85} we know that each element of $U(q)$ can be
expressed uniquely as 
\begin{equation}
S(a, b, c) = \alpha(a) \beta(b) \gamma(c)
\end{equation}
so $U(q) = \set{S(a, b, c) \mid a, b, c \in \F_q}$, and it
follows that $\abs{U(q)} = q^3$. Also, $U(q)$ is a Sylow
$3$-subgroup of $\Ree(q)$, and direct calculations show that 
\begin{align}
\label{ree_matrix_id1} \begin{split}
& S(a_1, b_1, c_1) S(a_2, b_2, c_2) = \\
& = S(a_1 + a_2, b_1 + b_2 - a_1 a_2^{3t}, c_1 + c_2 - a_2 b_1 + a_1 a_2^{3t + 1} - a_1^2 a_2^{3t}), 
\end{split} \\
& S(a, b, c)^{-1} = S(-a, -(b + a^{3t + 1}), -(c + ab - a^{3t + 2})), \\
\begin{split}
& S(a_1, b_1, c_1)^{S(a_2, b_2, c_2)} = \\
& = S(a_1, b_1 - a_1 a_2^{3t} + a_2 a_1^{3t}, c_1 + a_1 b_2 - a_2 b_1 + a_1 a_2^{3t + 1} - a_2 a_1^{3t + 1} - a_1^2 a_2^{3t} + a_2^2 a_1^{3t})
\end{split}
\end{align}
and
\begin{align} 
\label{ree_matrix_id2} S(a, b, c)^{h(\lambda)} &= S(\lambda^{3t - 2} a, \lambda^{1 - 3t} b, \lambda^{-1} c).
\end{align}

It follows that $\mathfrak{S} = \Ree(q) = \gen{S(1, 0, 0), h(\omega), \Upsilon}$, and these are our standard generators. The group preserves a symmetric bilinear form on $\F_q^7$, represented by the matrix
\begin{equation} \label{standard_bilinear_form}
J = \antidiag(1, 1, 1, -1, 1, 1, 1)
\end{equation}

From \cite{ward66}, \cite[Chapter $11$]{huppertIII} and \cite[Chapter $4.5$]{MR2562037} we obtain the following.
\begin{pr} \label{sylow3_props}
Let $G = \Ree(q)$.
\begin{enumerate}
\item $\abs{G} = q^3 (q^3 + 1) (q - 1)$ where $\gcd(q^3 + 1, q - 1) = 2$.
\item Conjugates of $U(q)$ intersect trivially.
\item The centre $\Zent(U(q)) = \set{S(0, 0, c) \mid c \in F_q}$.
\item The derived group $U(q)^{\prime} = \set{S(0, b, c) \mid b, c \in F_q}$, and its elements have order $3$.
\item The elements in $U(q) \setminus U(q)^{\prime} = \set{S(a, b, c) \mid a \neq 0}$ have order $9$ and their cubes form $\Zent(U(q)) \setminus \gen{1}$. 
\item $\Norm_G(U(q)) = U(q) H(q)$ and $G$ acts doubly transitively on the right cosets of $\Norm_G(U(q))$, \emph{i.e.} on a set of size $q^3 + 1$.
\item $U(q) H(q)$ is a Frobenius group with Frobenius kernel $U(q)$.
\item The proportion of elements of order $q-1$ in $U(q) H(q)$ is $\phi(q - 1) / (q - 1)$, where $\phi$ is the Euler totient function.
\item $\Norm_G(H(q)) = \gen{h(\omega), \Upsilon} \cong \Dih_{2(q-1)}$.
\end{enumerate}
\end{pr}

For our purposes, we want another set to act (equivalently) upon.

\begin{pr} \label{ree_doubly_transitive_action}
There exists $\OV \subseteq \PS^6(\F_q)$ on which $G = \Ree(q)$ acts faithfully and doubly transitively. Namely,
\begin{equation} \label{ree_ovoid_def}
\begin{split}
\OV = &\set{(0 : 0 : 0 : 0 : 0 : 0 : 1)} \cup \\
 &\lbrace (1 : a^t : -b^t : (ab)^t - c^t : -b - a^{3t + 1} - (ac)^t :  -c - (bc)^t - a^{3t + 2} - a^t b^{2t} : \\ 
& a^t c - b^{t + 1} + a^{4t + 2} - c^{2t} - a^{3t + 1} b^t - (abc)^t) \rbrace 
\end{split}
\end{equation}
Moreover, the stabiliser of $P_{\infty} = (0 : 0 : 0 : 0 : 0 : 0 : 1)$ is $U(q) H(q)$, the stabiliser of $P_0 = (1 : 0 : 0 : 0 : 0 : 0 : 0)$ is $(U(q) H(q))^{\Upsilon}$ and the stabiliser of $(P_{\infty}, P_0)$ is $H(q)$.
\end{pr}
\begin{proof} Notice that $\OV \setminus
  \set{P_{\infty}}$ consists of the first rows of the elements of $U(q)H(q)$.
  From \cite{ree85} we know that $G$ is the
  disjoint union of $U(q)H(q)$ and $U(q)H(q) \Upsilon U(q)H(q)$.
  Define a map between the $G$-sets as $(U(q)H(q))g \mapsto P_{\infty}
  g$.

  If $g \in U(q) H(q)$ then $P_{\infty} g = P_{\infty}$ and hence the
  stabiliser of $P_{\infty}$ is $U(q) H(q)$. If $g \notin U(q) H(q)$
  then $g = x \Upsilon y$ where $x,y \in U(q) H(q)$. Hence $P_{\infty}
  g = P_0 y \in \OV$ since $P_0 y$ is the first row of $y$. It follows
  that the map defines an equivalence between the $G$-sets.
\end{proof}



\begin{pr} \label{pr_involution_props}
Let $G = \Ree(q)$.
\begin{enumerate}
\item The stabiliser in $G$ of any two distinct points of $\OV$ is conjugate to $H(q)$. 
\item The stabiliser of any triple of distinct points has order at most $2$.
\item The number of elements in $G$ that fix exactly one point is $q^6 - 1$.
\item All involutions in $G$ are conjugate in $G$.
\item An involution fixes $q + 1$ points.
\end{enumerate}
\end{pr}
\begin{proof}
\begin{enumerate}
\item Immediate from \cite[Chapter $11$, Theorem $13.2$(d)]{huppertIII}.
\item The element $h(\lambda)$ can only fix a point $P \in \OV \setminus \set{P_{\infty}, P_0}$ if $\lambda = \pm 1$.
\item A stabiliser of a point is conjugate to $U(q) H(q)$, and there
  are $\abs{\OV}$ conjugates. The elements fixing exactly one point are the non-trivial elements of $U(q)$. Therefore the number of such elements is 
$\abs{\OV} (\abs{U(q)} - 1) = (q^3 + 1)(q^3  - 1) = (q^6 - 1)$.
\item Immediate from \cite[Chapter $11$, Theorem $13.2$(e)]{huppertIII}.
\item Each involution is conjugate to \[h(-1) = \diag(-1, 1,
  -1, 1, -1, 1, -1).\] Evidently, $h(-1)$ fixes $P_{\infty}$ since
  $h(-1) \in H(q)$. If $P = (p_1 : \dotsm : p_7) \in \OV$ with $p_1
  \neq 0$, then $P$ is fixed by $h(-1)$ if and only if $p_2 = p_4 = p_6 = 0$.
  But then $P$ is uniquely determined by $p_3$, so there are $q$
  possible choices for $P$. Thus the number of points fixed by $h(-1)$ is $q
  + 1$.
\end{enumerate}
\end{proof}

We shall need the following general result, whose easy proof we omit.

\begin{lem} \label{lem_1eigenvalue}
  Let $g \in G \leqslant \GL(d, F)$, where $d$ is odd, and $F$ a finite
  field, and assume that $G$ preserves a non-degenerate bilinear form
  and $\det(g) = 1$. Then $g$ has $1$ as an eigenvalue.
\end{lem}
  
  

\begin{pr} \label{cyclic_subgroups_conjugate}
All cyclic subgroups of $G = \Ree(q)$ of order $q - 1$ are conjugate to $H(q)$ and hence each is a stabiliser of two points of $\OV$.
\end{pr}
\begin{proof}
Let $C = \gen{g} \leqslant G$ be cyclic of order $q - 1$ and let $p$
  be an odd prime such that $p \mid q - 1$. Then there exists $k \in \Z$
  such that $\abs{g^k} = p$. Since $q^3 + 1 \equiv 2 \pmod p$ and $\abs{g^k} > 2$, the
  cycle structure of $g^k$ on $\OV$ must be a number of $p$-cycles and
  $2$ fixed points $P$ and $Q$. Since $G$ is doubly transitive there
  exists $x \in G$ such that $Px = P_{\infty}$ and $Qx = P_0$.
  
  Now either $g$ fixes $P$ and $Q$ or interchanges them, so $g^x \in
  \Norm_G(H(q)) = \gen{H(q), \Upsilon} \cong \Dih_{2(q-1)}$. Hence
  $\gen{g^x} = H(q)$, the unique cyclic subgroup of order
  $q - 1$ in $\gen{H(q), \Upsilon}$.
\end{proof}

\begin{pr} \label{ree_maximal_subgroup_list}
A maximal subgroup of $G = \Ree(q)$, with $q = 3^{2m+1}$ for some $m > 0$, is conjugate to one of the following subgroups:
\begin{itemize}
\item $\Norm_G(U(q)) = U(q) H(q)$, the point stabiliser,
\item $\Cent_G(j) \cong \gen{j} \times \PSL(2, q)$, the centraliser of an involution $j$,
\item $\Norm_G(A_0) \cong (\Cent_2 \times \Cent_2 \times A_0) {:} \Cent_6$, where $A_0 \leqslant \Ree(q)$ is cyclic of order $(q + 1) / 4$,
\item $\Norm_G(A_1) \cong A_1 {:} \Cent_6$, where $A_1 \leqslant \Ree(q)$ is cyclic of order $q + 1 - 3t$,
\item $\Norm_G(A_2) \cong A_2 {:} \Cent_6$, where $A_2 \leqslant \Ree(q)$ is cyclic of order $q + 1 + 3t$,
\item $\Ree(s)$ where $q$ is a proper power of $s$.
\end{itemize}
Moreover, all maximal subgroups except the last are reducible. 
\end{pr}
\begin{proof}
  The structure of the maximal subgroups follows from \cite{kleidman88} and \cite{ree85}. Hence it is sufficient to prove the final statement. 
  
  Clearly the point stabiliser is reducible. By Proposition \ref{pr_involution_props}, $j$ is conjugate to $h(-1) =
\diag{(-1, 1, -1, 1, -1, 1, -1)}$ so it has two eigenspaces $E_3$ and $E_4$ for $1$ and $-1$ respectively. Clearly $\dim E_3 = 3$ and $\dim E_4 = 4$.

Let $v \in E_3$ and $g \in \PSL(2, q)$. Then $(vg) j = (vj)g = vg$
since $g$ centralises $j$ and $j$ fixes $v$, which shows that $vg \in
E_3$, so this subspace is fixed by $\PSL(2, q)$. Similarly, $E_4$ is
also fixed. Hence $E_3$ and $E_4$ are submodules and the involution centraliser is reducible.

  Let $N$ be a normaliser of a cyclic subgroup and let $x$ be a
  generator of the cyclic subgroup that is normalised. Since $G < \SO(7, q)$, by Lemma
  \ref{lem_1eigenvalue}, $x$ has an eigenspace $E$ for the
  eigenvalue $1$, where $E$ is a proper non-trivial subspace of $V$. If $v \in E$ and
  $n \in N$, then $(vn) x^n = vn$ so that $vn$ is fixed by
  $\gen{x^n} = \gen{x}$. This implies that $vn \in E$ and thus $E$ is
  a proper non-trivial $N$-invariant subspace, so $N$ is reducible.
\end{proof}

\begin{pr} \label{pr_element_props}
Let $G = \Ree(q)$.
\begin{enumerate}
\item The centraliser of an involution $j \in G$ is isomorphic to $\gen{j} \times \PSL(2, q)$ and hence has order $q (q^2 - 1)$.
\item The number of involutions in $G$ is $q^2 (q^2 - q + 1)$.
\item The number of elements in $G$ of order $q - 1$ is $\phi(q - 1) q^3 (q^3 + 1) / 2$.
\item The number of elements in $G$ of even order is $q^2(7 q^5 - 23 q^4 + 8 q^3 + 23 q^2 - 39 q + 24) / 24$.
\item The number of elements in $G$ that fix at least one point is $q^2 (q^5 - q^4 + 3 q^2 - 5q + 2) / 2$
\end{enumerate}
\end{pr}
\begin{proof} 
\begin{enumerate}
\item Immediate from \cite[Chapter $11$]{huppertIII}.
\item All involutions are conjugate, and the index in $G$ of the involution centraliser is $\frac{q^3 (q^3 + 1) (q - 1)}{q (q^2 - 1)} = q^2 (q^2 - q + 1)$.

\item By Proposition \ref{cyclic_subgroups_conjugate}, each cyclic subgroup of order $q - 1$ is a stabiliser of two
  points and is uniquely determined by the pair of points that it
  fixes. Hence the number of cyclic subgroups of order $q - 1$ is
$\abs{\binom{\OV}{2}} = \frac{q^3 (q^3 + 1)}{2}$. By Proposition \ref{pr_involution_props}, the intersection of two distinct subgroups has order at most $2$, so the number of elements of order $q - 1$ is the number of generators of all these subgroups. 
\item By \cite[Lemma $2$]{ree85} and Proposition \ref{ree_maximal_subgroup_list}, every element of even order lies in
  a cyclic subgroup of order $q - 1$ or $(q + 1) / 2$. In each cyclic subgroup of order
  $q - 1$ there is a unique involution and hence $(q - 3) / 2$
  non-involutions of even order. Similarly, there are $(q - 3) / 4$ non-involutions in a cyclic subgroup of order $(q + 1) / 2$. By Proposition \ref{cyclic_subgroups_conjugate} the total number of elements of even
  order is therefore
\begin{multline}
(q - 3)(q^3 + 1) q^3 / 4 + (q - 3) (q - 1) (q^2 - q + 1) q^3 / 24 + q^2 (q^2 - q + 1) \\
= q^2(7 q^5 - 23 q^4 + 8 q^3 + 23 q^2 - 39 q + 24) / 24
\end{multline}
\item The only non-trivial elements of $G$ that fix more than $2$
  points are involutions. Hence in each cyclic subgroup of order $q -
  1$ there are $q - 3$ elements that fix exactly $2$ points, so by Proposition \ref{pr_involution_props}, the
  number of elements that fix at least one point is $q^6 + \frac{(q - 3)(q^3 + 1) q^3}{2} + q^2 (q^2 - q + 1) =  \frac{q^2 (q^5 - q^4 + 3 q^2 - 5q + 2)}{2}$.
\end{enumerate}
\end{proof}

\begin{pr} \label{pr_inv_cent_module_structure}
Let $G = \Ree(q)$ with natural module $V$, let $j \in G$ be an involution and let $C = \Cent_G(j)$.
\begin{enumerate}
\item $C^{\prime} \cong \PSL(2, q)$
\item $V\vert_{C^{\prime}} \cong V_3 \oplus V_4$ where $\dim V_i = i$.
\item $V_3$ and $V_4$ are absolutely irreducible. Moreover, $V_4 \cong \mathfrak{A}^{\varphi^i} \otimes \mathfrak{A}^{\varphi^k}$ and $V_3 \cong \mathcal{S}^2(\mathfrak{A})$ where $0 \leqslant i < k \leqslant 2m$.
\item When $j = h(-1)$, the forms preserved on $V_3$ and $V_4$ are $J_3 = \antidiag(1, -1, 1)$ and $J_4 = \antidiag(1, 1, 1, 1)$, up to scalar multiples.
\end{enumerate}
\end{pr}
\begin{proof}
\begin{enumerate}
\item Immediate from Proposition \ref{ree_maximal_subgroup_list}.
\item From the proof of Proposition \ref{ree_maximal_subgroup_list}, we see that $V\vert_C$ has submodules $V_3$ and $V_4$, so this is also true of $V\vert_{C^{\prime}}$, since $C^{\prime} \cong \PSL(2, q)$.
\item Let $C_3$ be the group acting on $V_3$. Since $[\Upsilon, h(\omega)] \in \Cent_G(h(-1))^{\prime}$ acts non-trivially on its corresponding $3$-dimensional submodule, $C_3$ is non-trivial. Hence $C_3 \cong \PSL(2, q)$ since $\PSL(2, q)$ is simple. If $V_3$ is reducible, it must have three $1$-dimensional constituents, since $\PSL(2, q)$ has no irreducible module of dimension $2$. Again since $\PSL(2, q)$ is simple, these constituents must be trivial, which is clearly false since $C_3$ is non-trivial. Therefore $V_3$ is irreducible.

Let $C_4$ be the group acting on $V_4$. Similarly, $C_4 \cong \PSL(2, q)$ and assume $V_4$ is reducible. Then $V_4 \cong 1 \oplus V_3^{\prime}$, where
$V_3^{\prime}$ is irreducible of dimension $3$ and the $1$-dimensional
module is trivial. This implies that every $g \in C_4$ has $1$ as an
eigenvalue. Again, $[\Upsilon, h(\omega)]$ provides a contradiction.

The result now follows from the structure of irreducible modules of $\PSL(2, q)$ \cite[§30]{MR0004042}.
\item Clearly, if $V = \gen{e_1, \dotsc, e_7}$ then $V_3 =
\gen{e_2, e_4, e_6}$ and $V_4 = \gen{e_1, e_3, e_5, e_7}$. The form $J$ then restricts to $J_4 = \antidiag(1, 1, 1, 1)$ on $V_4$ and $J_3 = \antidiag(1, -1, 1)$ on $V_3$.
\end{enumerate}
\end{proof}

\begin{pr} \label{pr_inv_cent_normaliser}
Let $G = \Ree(q)$ and let $C = \Cent_G(j)$ for some involution $j \in G$. Let $N = \Norm_{\Omega(7, q)}(C^{\prime})$. Then $[N : C^{\prime}] = 4$.
\end{pr}
\begin{proof}
By Proposition \ref{pr_inv_cent_module_structure}, $C^{\prime} \cong \PSL(2,
q)$, its module splits up
as a $3$-space and a $4$-space, and $C^{\prime}$ acts diagonally on these submodules. The normaliser must preserve this decomposition, and from Proposition \ref{pr_inv_cent_module_structure} it is clear that the form preserved on the $4$-space is of $+$-type, so $N$ embeds in
$\SO(3, q) \times \SO^{+}(4, q)$. Let $C_3$ and $N_3$ be the images of $C^{\prime}$ and $N$ on the $3$-space. Define $C_4$ and $N_4$ analogously.

Now $C_3 \cong \PSL(2, q)$ is in fact the natural representation of $\Omega(3, q)$, so $N_3 \cong \SO(3, q)$ and $[N_3 : C_3] = 2$. Since $\SO^{+}(4, q) \cong (\SL(2, q) \circ \SL(2, q)).2$ \cite[Chapter $4.5$]{MR2562037}, it acts as a tensor product on the
$4$-space. By Proposition \ref{pr_inv_cent_module_structure}, $C^{\prime}$ acts diagonally as a tensor product on the $4$-space. Clearly, the only element in $\SO^{+}(4, q) \setminus C_4$
which can normalise $C_4$ is the central element $-1$, so $[N_4 : C_4] = 2$. This proves the result.
\end{proof}

\begin{pr} \label{pr_inv_cent_forms}
Let $G = \Ree(q)$ and let $C = \Cent_G(j)$ where $j = h(-1) \in G$ is the diagonal involution. A non-degenerate symmetric bilinear form preserved by $C^{\prime}$ has a matrix representation with shape $\antidiag(b, a, b, -a, b, a, b)$ for some $a,b \in \F_q^{\times}$.
\end{pr}
\begin{proof}
By Proposition \ref{pr_inv_cent_module_structure}, the module $V$ of $C^{\prime}$ splits up as $V_3 \oplus V_4$ where
$\dim V_i = i$ and the preserved forms on these submodules are $J_3 = \antidiag(1, -1, 1)$ and $J_4 = \antidiag(1, 1, 1, 1)$. These are unique up to scalar multiples since the modules are
absolutely irreducible. Since $V_3$ and $V_4$ are also eigenspaces for
$j$, they must be orthogonal complements of each other for every form
preserved by $C^{\prime}$.

Hence the matrix of an arbitrary non-degenerate form $J_7$ on $V$, preserved by
$C^{\prime}$, must be the anti-diagonal join, with rows interchanged accordingly, of $a J_3$ and $b J_4$, for some $a, b \in \F_q^{\times}$. This proves the result.
\end{proof}

\begin{pr} \label{pr_conj_form}
Let $G$ be a conjugate of $\mathfrak{S} = \Ree(q)$ such that $C = \Cent_{\mathfrak{S}}(j)^{\prime} < G \cap \mathfrak{S}$ where $j = h(-1) \in \mathfrak{S}$ is the diagonal involution. A non-degenerate symmetric bilinear form preserved by $G$ has a matrix representation with shape $\antidiag(b, a, b, -a, b, a, b)$ for some $a,b \in (\F_q^{\times})^2$.
\end{pr}
\begin{proof}
Observe that $G$ preserves a symmetric bilinear form $K$ that is also preserved by $C$, so by Proposition \ref{pr_inv_cent_forms}, $K = \antidiag(b, a, b, -a, b, a ,b)$. Therefore $dJd = K$ where $d = \diag(d_1, \dotsc, d_7)$, $G < \Omega(7, q)^{d^{-1}}$ and $d$ normalises $C$. Observe that $\Upsilon \in C$, so it is centralised by $d$, and hence $d_1 = d_7$, $d_2 = d_6$ and $d_3 = d_5$. From $dJd = K$ we then see that $a = d_4^2$ and $b = d_1^2$ which proves the result.
\end{proof}

\begin{lem} \label{ree_totient_prop}
If $g \Ree(q)$ is uniformly random, then
\begin{align}
\Pr{\abs{g} = q - 1} = \frac{\phi(q - 1)}{2(q - 1)} &> \frac{1}{12\log{\log(q)}} \\
\Pr{\abs{g} \text{even}} = \frac{7q^2 - 9q - 24}{24 q (q + 1)} &> 1/4  
\end{align}
\begin{equation}
\Pr{g \ \text{fixes a point}} = \frac{-2 + 3 q + q^4}{2(q + q^4)} \geqslant 1/2
\end{equation}
\end{lem}
\begin{proof}
In each case, the first equality follows from Proposition \ref{pr_element_props} and Proposition \ref{sylow3_props}. In the first case, the
inequality follows from \cite[Section II.8]{totient_prop}, and in the other cases the inequalities are clear since $q \geqslant 27$.
\end{proof}

\begin{cl} \label{cl_random_selections} In $\Ree(q)$, the expected
  number of random selections required to obtain an element of order
  $q - 1$ is $\OR{\log \log q}$. Similarly, the expected number of random selections required obtain an element that fixes a point, or an element of
  even order, is $\OR{1}$.
\end{cl}
\begin{proof}
  Clearly the number of selections is geometrically distributed, where
  the success probabilities for each selection are given by Lemma
  \ref{ree_totient_prop}. Hence the expectations are as stated.
\end{proof}

\begin{pr} \label{ree_conjugacy_classes}
Elements in $\Ree(q)$ of order prime to $3$, with the same trace, are conjugate.
\end{pr}
\begin{proof}
  From \cite{ward66}, the number of conjugacy classes of non-identity
  elements of order prime to $3$ is $q - 1$. Observe that for $\lambda
  \in \F_q^{\times}$, $\Tr (S(0, 0, 1) \Upsilon h(\lambda)) = \lambda^t -
  1$ and $\abs{S(0, 0, 1) \Upsilon h(\lambda)}$ is prime to $3$ if also
  $\lambda \neq -1$.
  
  Moreover, $h(-1)$ has order $2$ and trace $-1$ so there are $q - 1$
  possible traces for non-identity elements of order prime to $3$, and
  elements with different trace must be non-conjugate. Thus all
  conjugacy classes must have different traces.
\end{proof}


\begin{pr} \label{psl_generation}
Let $G = \PSL(2, q)$. If $x, y \in G$ are uniformly random, then 
\begin{equation}
\Pr{\gen{x, y} = G} = 1 - \OR{\sigma_0(\log(q)) / q}
\end{equation}
\end{pr}
\begin{proof}
The maximal subgroup $M \leqslant G$ consisting of the upper triangular matrices modulo scalars has index $q + 1$, and all subgroups isomorphic to $M$ are conjugate. Since $M = \Norm_G(M)$, there are $q + 1$ conjugates of $M$. 
\begin{equation}
\Pr{\gen{x, y} \leqslant M^g \; \text{some} \; g \in G} \leqslant \sum_{i = 1}^{q + 1} \Pr{\gen{x, y} \leqslant M} = \frac{1}{q + 1}
\end{equation}
The other maximal subgroups have index strictly greater than $q + 1$. 

The number of conjugacy classes of
maximal subgroups is $\OR{\sigma_0(\log(q))}$, and hence the probability that $\gen{x,y}$ lies in a maximal subgroup is $\OR{\sigma_0(\log(q)) / q}$.
\end{proof}

The following result is analogous to \cite[Proposition 5.1]{baarnhielm05}. 
\begin{pr} \label{ree_prop_frobenius}
If $g_1, g_2 \in U(q)H(q)$ are uniformly random and independent, then
\begin{equation}
\Pr{\abs{[g_1, g_2]} = 9} = 1 - \frac{1}{q - 1}
\end{equation}
\end{pr}
\begin{proof}
  By Proposition \ref{sylow3_props}, $[g_1, g_2] \in U(q)$ and
  has order $9$ if and only if $[g_1, g_2] \notin U(q)^{\prime}
  \triangleleft U(q)H(q)$. It is therefore sufficient to find the
  proportion of (unordered) pairs $k_1, k_2 \in U(q) H(q) / U(q)^{\prime} =: \mathfrak{B}$ such that $[k_1, k_2] = 1$. 

  If $k_1 = 1$ then $k_2$ can be any element of $\mathfrak{B}$, which gives $q(q
  - 1)$ pairs.  If $1 \neq k_1 \in U(q) / U(q)^{\prime} \cong
  \F_q$ then $\Cent_{\mathfrak{B}}(k_1) = U(q) / U(q)^{\prime}$, so we again obtain
  $q(q - 1)$ pairs. Finally, if $k_1 \notin U(q)$ then
  $\abs{\Cent_{\mathfrak{B}}(k_1)} = q - 1$ so we obtain $q(q - 2)(q - 1)$ pairs.
  Thus we obtain $q^2 (q - 1)$ pairs from a total of $\abs{\mathfrak{B} \times \mathfrak{B}}
  =q^2 (q - 1)^2$ pairs, and the result follows.
\end{proof}

\subsection{Alternative definition} \label{section:alt_definition} The
definition of $\Ree(q)$ that we have given is the one that best suits
most of our purposes. However, to deal with some aspects of constructive recognition, we need the more common definition.

Following \cite[Chapter $4.5$]{MR2562037}, the exceptional group $\mathrm{G}_2(q)$ is
constructed by considering the Cayley algebra $\OO$ (the octonion
algebra), which has dimension $8$, and defining $\mathrm{G}_2(q)$ as the
automorphism group of $\OO$. Thus each element of $\mathrm{G}_2(q)$ fixes the
identity and preserves the algebra multiplication, and it follows that
$\mathrm{G}_2(q)$ is isomorphic to a subgroup of $\Omega(7, q)$.

Furthermore, when $q$ is an odd power of $3$, $\mathrm{G}_2(q)$ has a
certain outer automorphism, sometimes called the \emph{exceptional
  outer automorphism}, whose set of fixed points forms a group denoted
$\Ree(q) = \G2(q)$. The automorphism is defined in \cite[Chapter
4]{MR2562037}, as well as in \cite{MR2653666, MR2653247}, and amounts
to a mapping from the natural module to a section of the exterior
square, followed by the field automorphism $\varphi^m$.

\section{Algorithms}
In the following sections, we will describe the main algorithms.

In Section \ref{section:ree_conjugacy} we describe an algorithm that
takes a $\GL(7, q)$-conjugate of $\Ree(q)$ and finds a matrix that
conjugates it to $\Ree(q)$. Hence this is a constructive recognition
algorithm for the small Ree groups in the natural representations. As a
component of that algorithm, a (non-constructive) recognition
algorithm for $\Ree(q)$ is needed. We describe such an algorithm in
Section \ref{section:standard_recognition}.

In Section \ref{section:ree_constructive_membership} we describe an algorithm for constructive membership testing in $\Ree(q)$. That algorithm needs to find generators for the stabiliser of a point of $\OV$. An algorithm that accomplishes that is described in Section \ref{section:stabiliser_elements}.

\section{Recognition} \label{section:standard_recognition}

\begin{thm} \label{thm_ree_standard_recognition}
  There exists a Las Vegas algorithm that, given $\gen{X} \leqslant \GL(7,
  q)$, decides whether or not $\gen{X} = \Ree(q)$. The algorithm has expected time complexity $\OR{\sigma_0(\log(q))(\abs{X} + \log(q))}$ field
  operations.
\end{thm}
\begin{proof}
Let $G = \Ree(q)$. The algorithm proceeds as
follows:
\begin{enumerate}
\item Determine if $X \subseteq G$: all the following steps must succeed in order to conclude that a given $g
  \in X$ also lies in $G$.
\begin{enumerate}
\item Determine if $g \in \Omega(7, q)$, which is true if $\det g = 1$, if $g J g^T = J$, where $J$ is given by
  \eqref{standard_bilinear_form}, and if the spinor norm of $g$ is $1$. The spinor norm is calculated using \cite[Theorem 2.10]{MR2765375}.
\item Determine if $g \in \mathrm{G}_2(q)$, which from Section \ref{section:alt_definition} is true if $g$ preserves the
  octonion algebra multiplication $\cdot$. Hence test if $(e_i \cdot e_j)g = (e_ig) \cdot (e_jg)$ for each $i,j = 1, \dotsc, 7$, where $M = \gen{e_1, \dotsc, e_7}$ is the natural module of $\mathrm{G}_2(q)$. The multiplication table for $\cdot$ can be pre-computed using \cite[Chapter $4.5$]{MR2562037}.
\item Determine if $g$ is a fixed point of the exceptional outer
  automorphism of $\mathrm{G}_2(q)$, mentioned in Section
  \ref{section:alt_definition}. From the precise definition of the
  automorphism in \cite[Chapter $4.5$]{MR2562037}, it follows that
  computing the automorphism amounts to extracting a submatrix of the
  exterior square of $g$ and then mapping $g \mapsto \varphi^m(g)$.
\end{enumerate}

\item Determine if $\gen{X}$ is a proper subgroup of $G$, or
  equivalently if $\gen{X}$ is contained in a maximal subgroup. By
  Proposition \ref{ree_maximal_subgroup_list}, it is
sufficient to determine if $\gen{X}$ can be written over a smaller
field or if $\gen{X}$ is reducible. This can be done using the
  algorithms described in \cite{smallerfield} and the MeatAxe \cite{meataxe, better_meataxe}.
\end{enumerate}

The first step takes $\OR{\abs{X}\log(q)}$ field operations. The expected time of the algorithms in \cite{smallerfield} and of the MeatAxe is
$\OR{\sigma_0(\log(q))(\abs{X} + \log(q))}$ field operations. Hence our recognition algorithm has the stated expected time,
and it is Las Vegas since the MeatAxe is Las Vegas.
\end{proof}

\section{Finding an element of a stabiliser} \label{section:stabiliser_elements}

Let $G = \Ree(q) = \gen{X}$. The algorithm for constructive membership testing needs to obtain independent random elements of $G_P$, for a given point $P \in \OV$, as $\SLP$s in $X$. This is straightforward if, for any pair of points $P,Q
\in \OV$, we can construct $g \in G$ as an $\SLP$ in $X$ such that $Pg = Q$.


We first give an overview of the algorithm for accomplishing this. The general idea is to
obtain an involution $j \in G$ by random search, and then compute
$\Cent_G(j) \cong \gen{j} \times \PSL(2, q)$ using \cite{bray00}. The
given $G$-module restricted to the centraliser splits up as in Proposition
\ref{pr_inv_cent_module_structure}, and the points $P,Q \in \OV$
project to points $P_3, Q_3$ in the $3$-dimensional submodule. If the
projections satisfy certain conditions, then we can write down $g \in
\Cent_G(j)$ that maps $P_3$ to $Q_3$, and obtain $g$ as an $\SLP$ in
the generators of $\Cent_G(j)$ using the maps from Theorem
\ref{thm_psl_recognition}. With high probability, we can then multiply
$g$ by an element that fixes $P_3$ so that it also maps $P$ to $Q$. A
discrete logarithm oracle is needed in that step. When using
\cite{bray00}, we can easily keep track of $\SLP$s of the centraliser
generators, hence we obtain $g$ as an $\SLP$ in $X$.


By Corollary
\ref{cl_random_selections} it is easy to find elements of even order
by random search, which we can power up to obtain involutions.

To use \cite{bray00} we need an algorithm that
determines if the whole centraliser has been generated. Since its derived group should be
$\PSL(2, q)$, by Proposition \ref{psl_generation}, with high
probability it is sufficient to compute two random elements of the
derived group. Random elements of the derived group can be obtained as
described in Section \ref{section:random_elements}.




Let us now describe the algorithm in more detail. First we fix some notation for the remainder of this Section.
\begin{itemize}
\item $j \in G = \gen{X} = \Ree(q)$ is an involution, and $C = \Cent_G(j) = \gen{Y}$,
\item $V \cong V_3 \oplus V_4$ is the module of $C^{\prime}$
\item $\zeta_V : V \to V_3$ is the natural projection homomorphism,
\item $\zeta_{\OV} : \PS(V) \to \PS(V_3)$ is the induced projective map,
\item $\zeta_G : C \to \GL(3, q)$ is the corresponding group epimorphism. Define $C_3 = \zeta_G(C^{\prime})$.
\item $\pi_3 : \PSL(2, q) \to \Omega(3, q)$ is the symmetric square map (an isomorphism), so $\pi_3 : g \mapsto \mathcal{S}^2(g)$,
\item $\rho_G : C_3 \to \PSL(2, q)$ is the map to the standard copy from Theorem \ref{thm_psl_recognition},
\item $\pi_7 : C_3 \to C^{\prime}$ is calculated by first using $\rho_G$ to map an element to the standard copy, then expressing it as an $\SLP$, which is then evaluated on $Y$. 
\item $c_3$ is a change-of-basis from $V_3$ to $\mathcal{S}^2(\mathfrak{A})$. Hence $C_3^{c_3} = \IM \pi_3$.
\item Identify $\mathfrak{A}$ as $\gen{x} \oplus \gen{y}$ for indeterminates $x,y$.
\end{itemize}

Clearly, an application of the MeatAxe on $V$ provides a change-of-basis which allows us to set up the maps $\zeta_V$, $\zeta_{\OV}$ and $\zeta_G$. An application of Theorem \ref{thm_psl_recognition} on $C_3$ allows us to set up the maps $\pi_3$ and $\pi_7$, and to obtain $c_3$.

\subsection{Constructing a mapping element}
We now consider the algorithm that constructs elements that map one point
of $\OV$ to another. Since $\mathfrak{A} \cong \gen{x} \oplus \gen{y}$ we can identify the module $\mathfrak{\hat{U}} = \PS(\mathcal{S}^2(\mathfrak{A}))$ with
the space of quadratic forms in $x$ and $y$ modulo scalars, so that $\mathfrak{\hat{U}} = \PS(\gen{x^2}
\oplus \gen{xy} \oplus \gen{y^2})$. Then $C_3^{c_3}$ acts projectively on $\mathfrak{\hat{U}}$ and $\abs{\mathfrak{\hat{U}}} = \abs{\PS(V_3)} = \abs{\PS^2(\F_q)} = (q^3 - 1) / (q - 1) = q^2 + q + 1$. 

\begin{pr} \label{pr_psl_ovoid_action}
Under the action of $\mathfrak{\hat{H}} = C_3^{c_3}$, the set $\mathfrak{\hat{U}}$ splits into $3$ orbits.
\begin{enumerate}
\item The orbit containing $xy$, \emph{i.e.} the non-degenerate quadratic forms that represent $0$, which has size $q(q + 1) / 2$.
\item The orbit containing $x^2 + y^2$, \emph{i.e.} the non-degenerate quadratic forms that do not represent $0$, which has size $q(q - 1) / 2$.
\item The orbit containing $x^2$ (and $y^2$), \emph{i.e.} the degenerate quadratic forms, which has size $q + 1$.
\end{enumerate}
The pre-image in $\SL(2, q)$ of $\rho_G((\mathfrak{\hat{H}}_{xy})^{c_3^{-1}})$ is dihedral of order $2(q - 1)$, generated by the matrices
\begin{flalign} \label{sl2_stab_gens}
\begin{bmatrix}
\omega & 0 \\
0 & \omega^{-1} 
\end{bmatrix} &
\ \text{and} \ \begin{bmatrix}
0 & 1 \\
-1 & 0 
\end{bmatrix} 
\end{flalign}
\end{pr}
\begin{proof}
This is elementary theory of quadratic forms, except that we work projectively.

\end{proof}

\begin{pr} \label{pr_ree_ovoid_split}
Use the notation above.
\begin{enumerate}
\item The number of points of $\OV$ that are contained in $\Ker(\zeta_V)$ is $q + 1$.
\item Let $P \in \OV$ be uniformly random. The probability that $P \nsubseteq \Ker(\zeta_V)$, and that $\zeta_{\OV}(P) c_3$ is both non-degenerate and represents $0$ is at least $1/2 + \OR{1/q}$.

\end{enumerate}
\end{pr}
\begin{proof}
\begin{enumerate}
\item The map $\zeta_V$ projects onto $V_3$, so the kernel consists of
  those vectors that lie in $V_4$. From the proof of Proposition
  \ref{ree_maximal_subgroup_list}, with respect to a suitable basis,
  $V_4$ is the $-1$-eigenspace of $h(-1)$. Hence by an argument similar to the proof of Proposition
  \ref{pr_involution_props}, a point $P = (p_1, \dotsc, p_7) \in V_4$ if $p_2 = p_4 = p_6 = 0$ and there are $q+1$ such points in $\OV$.

\item Since $P$ is uniformly random and chosen independently of $\zeta_{\OV}$, it follows that $\zeta_{\OV}(P)$ is
  uniformly random from $\zeta_{\OV}(\OV)$. Without loss of
  generality we can take $c_3$ to be the identity. Using the notation
  above, $\zeta_{\OV}(P) = p_2 x^2 + p_4 xy + p_6 y^2$, which is
  degenerate if $p_2 = 0$. This
  happens with probability $(q + 1) / (q^2 + 1)$. If $p_2 \neq 0$, then $P \nsubseteq \Ker(\zeta_V)$, and we can then express the point as $\zeta_{\OV}(P) = x^2 + b xy + c y^2$
  where $(b, c)$ is uniformly distributed in $\F^2_q$. It is degenerate if $b^2 - c = 0$, which happens with probability
  $1/q$. If it is not degenerate, it represents $0$ when $b^2 - c$ is a
  square in $\F_q$, which happens with probability $1/2$. The
  result follows.
%
\end{enumerate}
\end{proof}


The algorithm that maps one point to another is given as Algorithm \ref{alg:find_mapping_element}.

\begin{figure}
\begin{codebox}
\refstepcounter{algorithm}
\label{alg:find_mapping_element}
\Procname{\kw{Algorithm} \ref{alg:find_mapping_element}: $\proc{FindMappingElement}(X, P, Q)$}
\li \kw{Input}: Generating set $X$ for $G = \Ree(q)$, $\mathfrak{H} = C_3$.
\zi Points $P \neq Q \in \OV$ such that $P, Q \nsubseteq \Ker(\zeta_V)$, and $\zeta_{\OV}(P)$ and $\zeta_{\OV}(Q)$ are 
\zi non-degenerate and represent $0$. 
\li \kw{Output}: $g_2 \in G$, written as an $\SLP$ in $X$, such that $Pg_2 = Q$.
\li $P_3 := \zeta_{\OV}(P) c_3 $; $Q_3 := \zeta_{\OV}(Q) c_3$ 
\li Construct upper triangular $g \in \PSL(2, q)$ such that $P_3 \pi_3(g) = Q_3$ \label{find_mapping_element_tria_map}
\li     $R_3 := \zeta_{\OV}(P \pi_7(\pi_3(g)^{c_3^{-1}})) c_3$ 
\li     \Comment Now $R_3 = Q_3$
\li     Construct $c \in \GL(3, q)$ such that $(xy) c = R_3$ \label{find_mapping_element_zero_map}
\li     Let $\mathfrak{D}$ be the image in $\PSL(2, q)$ of the diagonal matrix in \eqref{sl2_stab_gens} 
\li     $s := \pi_7(\pi_3(\mathfrak{D})^{c c_3^{-1}})$
\li     \Comment Now $\gen{s} \leqslant \zeta_G^{-1}(\mathfrak{H}_{R_3})$
\li     $\delta, z := \proc{Diagonalise}(s)$ 
\li     \Comment Now $\delta = s^z$
\li     \If $\exists \, \lambda_0 \in \F_q^{\times}$ such that $(P \pi_7(\pi_3(g)^{c_3^{-1}}) z) h(\lambda_0) = Q z$ \label{find_mapping_element_diag_map}
\zi     \Then
\li          $k := \proc{DiscreteLog}(\delta, h(\lambda_0))$ 
\li          \Comment Now $\delta^k = h(\lambda_0)$
\li          \Return $\pi_7(\pi_3(g)^{c_3^{-1}}) s^k$ \label{find_mapping_element_ret}
        \End
\zi     \kw{end}
\li \Return \const{fail}
\end{codebox}
\end{figure}

\subsection{Constructing a stabilising element} \label{section:find_stab_element}
Let $G = \Ree(q) = \gen{X}$, $P \in \OV$ be given. The complete algorithm that constructs a random element of $G_P$ proceeds as follows.
\begin{enumerate}
\item Find a random involution $j \in G$.
\item Compute generators $Y$ for $C = \Cent_G(j)$ using \cite{bray00}, and generators for $C^{\prime}$ as described in Section \ref{section:random_elements}.
\item Use the MeatAxe to verify that the module for
  $C^{\prime}$ splits up only as in Proposition
  \ref{pr_inv_cent_module_structure}. 
\item Return to the first step if $P$ lies in the kernel of $\zeta_V$, if $\zeta_{\OV}(P) c_3$ is degenerate, or if it does not represent $0$.
\item Use Theorem \ref{thm_psl_recognition} to verify that we have the whole of $C^{\prime}$ and to set up maps listed at the start of of Section \ref{section:stabiliser_elements}. Return to the second step if this fails.
\item Take random $g_1 \in G$ and let $Q = Pg_1$. Repeat until $P \neq Q$, $Q$ does not lie in the kernel of $\zeta_V$ and $\zeta_{\OV}(Q) c_3$ is not degenerate and represents $0$.
\item Use Algorithm \ref{alg:find_mapping_element} to find $g_2 \in C^{\prime}$ such that $P g_2 = Q$. Return to the previous step if it fails, otherwise return $g_1 g_2^{-1}$.
\end{enumerate}

\subsection{Correctness and complexity}

\begin{lem} \label{lem_tri_map} Let $P_3, Q_3 \in \zeta_{\OV}(\OV)$ be non-degenerate and represent $0$. There exists $g \in \PSL(2, q)$ such
 that the pre-image of $g$ in $\SL(2, q)$ is upper triangular and $P_3
  \pi_3(g)^{c_3^{-1}} = Q_3$.
\end{lem}
\begin{proof}
Without loss of generality, we can take $c_3 = 1$. Since $P_3$ and $Q_3$ are non-degenerate, $P_3 = x^2 + a xy + b y^2$ and $Q_3 = x^2 + l xy + n y^2$ where $(1 : a : b)$ and $(1 : l : n)$ are in $\PS^2(\F_q)$. Also,
\begin{equation}
g = \begin{bmatrix}
u & v \\
0 & 1/u
\end{bmatrix}
\end{equation}
where $u, v \in \F_q$ and $u \neq 0$.

We want to determine $u, v$ such that $P \pi_3(g) = Q$. Note that
$g$ is the pre-image in $\SL(2, q)$ of an element in $\PSL(2, q)$ and
therefore $\pm u$ determine the same element of $\PSL(2, q)$. The map $\pi_3$ is the symmetric square map, and $\chr{\F_q} = 3$, so
\begin{equation}
\pi_3(g) = \begin{bmatrix}
u^2 & -uv & v^2 \\
0 & 1 & v/u \\
0 & 0 & 1/u^2
\end{bmatrix}
\end{equation}

This leads to the following equations:
\begin{align}
u^2 &= \mathfrak{C} \label{tri_map_eq1} \\
-uv + a &= \mathfrak{C} l \label{tri_map_eq2} \\
v^2 + a v u^{-1} + b u^{-2} &= \mathfrak{C} n \label{tri_map_eq3}
\end{align}
for some $\mathfrak{C} \in \F_q^{\times}$. We can solve for $u$ in
\eqref{tri_map_eq1} and for $v$ in \eqref{tri_map_eq2}, so that
\eqref{tri_map_eq3} becomes
\begin{equation}
\mathfrak{C}^2 (n - l^2) + a^2 - b = 0
\end{equation}
This quadratic equation has a solution if the discriminant $(l^2 - n) (a^2 - b) \in (\F_q^{\times})^2$. But the latter is true since both $(a^2 - b)$ and $(l^2 - n)$ are non-zero squares. The result follows.
\end{proof}

\begin{thm}  \label{thm_find_mapping_element_pr}
If Algorithm $\ref{alg:find_mapping_element}$ returns an element $g_2$, then $P g_2 = Q$. If $P$ and $Q$ are defined as in Section $\ref{section:find_stab_element}$, then the probability that the choice of $j$ results in Algorithm $\ref{alg:find_mapping_element}$ finding such an element is at least $1/2 + \OR{1/q}$.
\end{thm}
\begin{proof}
  By Proposition \ref{pr_psl_ovoid_action}, the point $R_3$ is in the
  same orbit as $xy$, so the element $c$ at line
  \ref{find_mapping_element_zero_map} can easily be found by
  diagonalising the form corresponding to $R_3$. Let $\mathfrak{H} = C_3$. Then $\pi_3(\mathfrak{D})^{c c_3^{-1}} \in \mathfrak{H}_{R_3}$ has order $(q - 1) / 2$. Hence $s$ also has order $(q - 1)/2$, and $s \in \zeta_G^{-1}(\mathfrak{H}_{R_3})$.
  
  We choose $Q$ such that there exists $g_1 \in C^{\prime}$ with $Pg_1
  = Q$. If we let $g_3 = \pi_3(g)^{c_3^{-1}}$, with $g$ as in the algorithm, and $R = P \pi_7(g_3)$, then $R \pi_7(g_3)^{-1} g_1 = Q$ and $\zeta_{\OV}(R) =
  R_3 = Q_3$. Hence $\zeta_G(\pi_7(g_3)^{-1} g_1) \in \mathfrak{H}_{Q_3}$, and
  therefore $\pi_7(g_3)^{-1} g_1 \in \zeta_G^{-1}(\mathfrak{H}_{R_3})$.
  
  By Proposition \ref{pr_psl_ovoid_action}, $\zeta_G^{-1}(\mathfrak{H}_{R_3})$ is
  dihedral of order $q - 1$, and $s$ generates a subgroup of index
  $2$. Therefore $\Pr{\pi_7(g_3)^{-1} g_1 \in \gen{s}} = 1/2$, which is
  the success probability of line \ref{find_mapping_element_diag_map}.
  
  It is straightforward to determine if $\lambda_0$ exists, since
  $h(\lambda_0)$ is diagonal. 
Hence the success probability of the algorithm is
  as stated.
\end{proof}

\begin{thm} \label{thm_find_mapping_element_comp}
  Assume an oracle for the discrete logarithm problem in $\F_q$.
The time complexity of Algorithm $\ref{alg:find_mapping_element}$ is $\OR{\log(q)^3 + \chi_D}$ field operations. The length of the returned $\SLP$ is $\OR{\log(q) \log\log(q)}$.
\end{thm}
\begin{proof}
  By Lemma \ref{lem_tri_map}, line \ref{find_mapping_element_tria_map}
  involves solving a quadratic equation in $\F_q$, and hence uses
  $\OR{1}$ field operations. Evaluating $\pi_3$ uses $\OR{1}$ field operations, and $\pi_7$
  uses $\OR{\log(q)^3}$ field operations. It is clear that the rest
  of the algorithm can be done using $\OR{\chi_D}$ field operations.

  By Theorem \ref{thm_psl_recognition}, the length of the $\SLP$ from the
  constructive membership testing in $\PSL(2, q)$ is $\OR{\log(q) \log\log(q)}$,
  which is therefore also the length of the returned $\SLP$.
\end{proof}

\begin{cl} \label{cl_find_stab_element} Assume an oracle for the discrete logarithm problem in $\F_q$. There exists a Las Vegas algorithm
  that, given $\gen{X} \leqslant \GL(7, q)$ such that $G = \gen{X} =
  \Ree(q)$ and $P \in \OV$, constructs a random element of $G_P$ as
  an $\SLP$ in $X$. 
The expected time complexity of the algorithm is
  $\OR{\xi  \log\log(q)  + \log(q)^3 + \chi_D}$ field operations. The length of the returned
  $\SLP$ is $\OR{\log(q) \log\log(q)}$.
\end{cl}
\begin{proof}
  The algorithm is given in Section \ref{section:find_stab_element}.
  
  An involution is found by obtaining a random element of even order and then raising it to an appropriate power. Hence by Corollary
  \ref{cl_random_selections}, the expected time to find an involution
  is $\OR{\xi + \log(q) \log\log(q)}$ field operations.

  By \cite[Theorem 7]{ryba_trick}, we can use \cite{bray00} to
  obtain generators of the centraliser, using $\OR{1}$ field
  operations. As described in Section \ref{section:random_elements},
  we can obtain uniformly random elements of its derived group. By
  Proposition \ref{psl_generation}, two random elements will generate
  $\PSL(2, q)$ with high probability. This implies that the expected
  time to obtain generators for $\PSL(2, q)$ is $\OR{1}$
  field operations.
  
  By Proposition \ref{pr_psl_ovoid_action}, $Q$ is equal to $P$ with
  probability $2/(q(q+1))$. In the algorithm in Section
  \ref{section:find_stab_element}, $P$ is given and $j$ is uniformly
  random, but Proposition \ref{pr_ree_ovoid_split} applies,
  since $\zeta_{\OV}(P)$ will then still be uniformly random. Hence $P$ has
  the required properties, with respect to the choice of $j$, with
  probability $1/2$, and similarly for $Q$, so the expected time of
  the penultimate step is $\OR{1}$ field operations.

  Note that $g_1$ is chosen randomly, and the map $g_1 \mapsto g_1 g_2^{-1}$ is deterministic (when $j$ is fixed) and independent of the choice of $g_1$. Therefore the output of the map, \emph{i.e.} the element
  returned by the algorithm in Section \ref{section:find_stab_element},
  is uniformly random in $G_P$.
  
  The expected time complexity of the last step is given by Theorem
  \ref{thm_find_mapping_element_pr} and \ref{thm_find_mapping_element_comp}. It follows from the above
  argument and from \cite{meataxe} and Theorem \ref{thm_psl_recognition} that the
  expected time complexity of the algorithm in Section
  \ref{section:find_stab_element} is as stated.

  The algorithm is clearly Las Vegas, since it
  is straightforward to check that the element we compute fixes
  the point $P$.
\end{proof}



\section{Constructive membership testing}
\label{section:ree_constructive_membership}

We now describe the constructive membership algorithm for our standard
copy $\Ree(q)$. Given a set of generators $X$, such that $G = \Ree(q) = \gen{X}$, and given $g \in G$, we want to express $g$ as an
$\SLP$ in $X$. To decide if $g \in G$, we use the first
step of the algorithm in Theorem \ref{thm_ree_standard_recognition}.

The general structure of the algorithm mirrors the corresponding
algorithm for the Suzuki groups \cite{baarnhielm05}. It consists of a
pre-processing step and a main step. The pre-processing step is only
executed once for a given $X$, and computes certain sets of matrices
necessary for the execution of the main step. The actual constructive
membership algorithm is the main step, which expresses a given $g \in
G$ as an $\SLP$ in $X$.

\subsection{Pre-processing} \label{section:preprocessing_step} The
main step requires the ability to express
elements of $\mathrm{O}_3(G_{P_{\infty}})$ and $\mathrm{O}_3(G_{P_0})$
as $\SLP$s in the given generators. Our approach for that resembles
Gaussian elimination, and the purpose of the pre-processing step is to construct sets of matrices that facilitate this.

For $\mathrm{O}_3(G_{P_{\infty}}) = U(q)$ these matrices are
\begin{equation} \label{inf_std_generators}
\set{S(a_i, x_i, y_i)}_{i = 1}^{n} \cup \set{S(0, b_i, z_i)}_{i = 1}^{n} \cup \set{S(0, 0, c_i)}_{i = 1}^n
\end{equation}
where $\set{a_1, \dotsc, a_n}$,
$\set{b_1, \dotsc, b_n}$, $\set{c_1, \dotsc, c_n}$ must form vector space bases of $\F_q$ over $\F_3$
(so $n = \log_3{q} = 2m + 1$), but are otherwise arbitrary. The elements $x_i, y_i, z_i \in \F_q$ are arbitrary. We shall denote any set of the form \eqref{inf_std_generators} as \lq\lq standard generators'' for $G_{P_{\infty}}$. Standard generators for $G_{P_0}$ are defined analogously.


The precise operations which need the standard generators, and hence the motivation for their definition, are given by following elementary result, which we state without proof.

\begin{lem} \label{ree_row_operations}
There exist algorithms for the following operations.
\begin{enumerate}
\item Given $g = h(\lambda) S(a, b, c) \in G_{P_{\infty}}$, construct $x \in \mathrm{O}_3(G_{P_{\infty}})$ expressed as an $\SLP$ in the standard generators, such that $gx = h(\lambda)$.
\item Given $g = S(a, b, c) h(\lambda) \in G_{P_{\infty}}$, construct $x \in \mathrm{O}_3(G_{P_{\infty}})$ expressed as an $\SLP$ in the standard generators, such that $xg = h(\lambda)$.
\item Given $P_{\infty} \neq P \in \OV$, construct $g \in \mathrm{O}_3(G_{P_{\infty}})$ expressed as an $\SLP$ in the standard generators, such that $Pg = P_0$.
\end{enumerate}
The $\SLP$s of the constructed elements have length $\OR{\log(q)}$. The algorithms have time complexity $\OR{\log(q)^3}$ field operations. Analogous algorithms exist for $G_{P_0}$.
\end{lem}

\begin{thm} \label{thm_pre_step} Given an oracle for the discrete logarithm problem in $\F_q$, the
  pre-processing step is a Las Vegas algorithm that constructs standard
  generators for $G_{P_{\infty}}$ and
  $G_{P_0}$ as $\SLP$s in $X$ of length $\OR{\log(q) (\log\log(q))^2}$. It has expected time complexity $\OR{(\xi \log\log(q) + \log(q)^3 + \chi_D) \log\log(q)}$ field operations. 
\end{thm}
\begin{proof}
The pre-processing algorithm consists of the following steps:
\begin{enumerate}
\item Obtain random $\mathfrak{a}_1, \mathfrak{a}_2 \in G_{P_{\infty}}$ and $\mathfrak{b}_1, \mathfrak{b}_2
  \in G_{P_0}$ using the algorithm from Corollary
  \ref{cl_find_stab_element}. Let $\mathfrak{c}_1 = [\mathfrak{a}_1, \mathfrak{a}_2]$, $\mathfrak{c}_2 = [\mathfrak{b}_1, \mathfrak{b}_2]$. 
\item Determine if there exists $\mathfrak{d}_1 \in \set{\mathfrak{a}_1, \mathfrak{a}_2}$ that can be diagonalised to $h(\lambda_0) \in
G$, for some $\lambda_0 \in \F_q^{\times}$ that does not lie in a proper subfield of $\F_q$. Similarly determine existence of a $\mathfrak{d}_2$ from $\mathfrak{b}_1$ and $\mathfrak{b}_2$. Determine if $\abs{\mathfrak{c}_1} = \abs{\mathfrak{c}_2} = 9$. Return to the first step if any of these tests fail.

\item 
As standard generators for $G_{P_{\infty}}$ we take $\mathfrak{U} = \mathfrak{U}_1 \cup \mathfrak{U}_2$ where 
\begin{eqnarray}
\mathfrak{U}_1 & =  & \bigcup_{i = 1}^{2m + 1} \set{\mathfrak{c}_1^{\mathfrak{d}_1^i}, (\mathfrak{c}_1^3)^{\mathfrak{d}_1^i}} \\ 
\mathfrak{U}_2 & = & \bigcup_{1 \leqslant i < j \leqslant 2m + 1} \set{[\mathfrak{c}_1^{\mathfrak{d}_1^i}, \mathfrak{c}_1^{\mathfrak{d}_1^j}]} 
\end{eqnarray}
From $\mathfrak{c}_2$ and $\mathfrak{d}_2$ we similarly obtain standard generators $\mathfrak{L}$ for $G_{P_0}$
\end{enumerate}

It follows from \eqref{ree_matrix_id1} and \eqref{ree_matrix_id2} that $\mathfrak{U}$
is of the form \eqref{inf_std_generators}. Similarly, $\mathfrak{L}$ has the correct form. Since the $\mathfrak{a}_i$ and $\mathfrak{b}_i$ are expressed as $\SLP$s in $X$, this is also true for the elements of $\mathfrak{U}$ and $\mathfrak{L}$.


By Corollary \ref{cl_find_stab_element}, the expected time to find
$\mathfrak{a}_i$ and $\mathfrak{b}_i$ is $\OR{\xi \log\log(q) + \log(q)^3 + \chi_D}$, and
these are uniformly distributed independent random elements. The
elements of order dividing $q - 1$ can be diagonalised as required.
By Proposition \ref{sylow3_props}, the proportion of elements of
order $q - 1$ in $G_{P_{\infty}}$ and $G_{P_0}$ is $\phi(q - 1) / (q - 1)$. 

It is straightforward to determine if $\mathfrak{a}_i$ or $\mathfrak{b}_i$ diagonalise to some $h(\lambda_0)$, since they are triangular. To determine if $\lambda_0$ lies
in a proper subfield, it is sufficient to determine if $\abs{\lambda_0} \mid
3^n - 1$, for some proper divisor $n$ of $2m + 1$. 

Hence by Proposition \ref{ree_prop_frobenius} the expected time for
the first two steps is \[\OR{(\xi \log\log(q) + \log(q)^3 + \chi_D) \log\log(q)}\] field
operations.




Since $\mathfrak{U}$ has the form \eqref{inf_std_generators}, as described in the beginning of Section \ref{section:preprocessing_step}, it determines three sets of
field elements $\set{a_1, \dotsc, a_{2m+1}}$, $\set{b_1, \dotsc,
  b_{2m+1}}$ and $\set{c_1, \dotsc, c_{2m+1}}$. By \eqref{ree_matrix_id2},
in this case each $a_i = a \lambda_0^i$, $b_i = b \lambda_0^{i(t + 2)}$
and $c_i = c \lambda_0^{i(t + 3)}$, for some $a,b,c \in
\F_q^{\times}$ depending on $\mathfrak{a}_i$ and $\mathfrak{b}_i$. Since
$\lambda_0$ does not lie in a proper subfield, these sets form vector
space bases of $\F_q$ over $\F_3$. Hence $\mathfrak{U}$ and $\mathfrak{L}$ are standard generators and the algorithm is Las Vegas.

\end{proof}

\subsection{Main algorithm}
We now present the algorithm to express an arbitrary $g \in G$ as an $\SLP$. It is given as Algorithm \ref{alg:ree_element_to_slp}.

The idea behind the algorithm is to make use of Lemma
\ref{ree_row_operations} to express elements in $G_{P_{\infty}}$ and
$G_{P_0}$ as $\SLP$s in the given generators. The Lemma effectively
performs Gaussian elimination to reduce an element to a diagonal
matrix $h(\lambda_0)$ for some $\lambda_0 \in \F_q$. This could then
be expressed as an $\SLP$ in the given generators using discrete
log. However, it turns out that we can also use the Lemma to
construct a diagonal matrix with the same trace as $h(\lambda_0)$, which by
Proposition \ref{ree_conjugacy_classes} must be a conjugate. Since
both lie in the same cyclic subgroup they must in fact be the
same matrix, or inverses. Hence discrete log can be avoided.

We can only use Lemma \ref{ree_row_operations} on elements that fix a point of $\OV$, so the first step is to multiply $g$ by a random element until it fixes a point, which happens with high probability.

\begin{figure}[ht]
\begin{codebox}
\refstepcounter{algorithm}
\label{alg:ree_element_to_slp}
\Procname{\kw{Algorithm} \ref{alg:ree_element_to_slp}: $\proc{ElementToSLP}(\mathfrak{U}, \mathfrak{L}, g)$}
\li \kw{Input}: Standard generators $\mathfrak{U}$ for $G_{P_{\infty}}$ and $\mathfrak{L}$ for $G_{P_0}$. Matrix $g \in \gen{X} = G$.
\li \kw{Output}: $\SLP$ for $g$ in $X$
\li \Repeat
\li     \Repeat
\li         $r := \proc{Random}(G)$ 
\li     \Until{$gr$ has an eigenspace $Q \in \OV$  \label{main_alg_find_point_ree}}
\li     Construct $\mathfrak{z}_1 \in G_{P_{\infty}}$ using Lemma \ref{ree_row_operations} and $\mathfrak{U}$ such that $Q\mathfrak{z}_1 = P_0$. \label{main_alg_row_op1_ree} 
\li     \Comment Now $(gr)^{\mathfrak{z}_1} \in G_{P_0}$
\li     Construct $\mathfrak{z}_2 \in G_{P_0}$ using Lemma \ref{ree_row_operations} and $\mathfrak{L}$ such that $(gr)^{\mathfrak{z}_1} \mathfrak{z}_2 = h(\lambda_0)$
\zi     for some $\lambda_0 \in \F_q^{\times}$ \label{main_alg_row_op2_ree} 
\li     $x := \Tr(h(\lambda_0))$ 
\li \Until $x - 1$ is a square in $\F_q^{\times}$ \label{main_alg_square_test}
\li     \Comment{Express diagonal matrix as $\SLP$}
\li Construct $\mathfrak{u} = S(0, 0, \sqrt{(x - 1)^{3t}})S(0, 1, 0)^{\Upsilon}$ using Lemma \ref{ree_row_operations} and $\mathfrak{U} \cup \mathfrak{L}$ \label{main_alg_row_op3_ree} 
\li \Comment{Now $\Tr(\mathfrak{u}) = x$}
\li Let $P_1, P_2 \in \OV$ be the fixed points of $\mathfrak{u}$
\li Construct $\mathfrak{a} \in G_{P_{\infty}}$ using Lemma \ref{ree_row_operations} and $\mathfrak{U}$ such that $P_1 \mathfrak{a} = P_0$ \label{main_alg_row_op4_ree} 
\li Construct $\mathfrak{b} \in G_{P_0}$ using Lemma \ref{ree_row_operations} and $\mathfrak{L}$ such that $(P_2 \mathfrak{a})\mathfrak{b} = P_{\infty}$ \label{main_alg_row_op5_ree} 
\li \Comment{Now $\mathfrak{u^{ab}} \in G_{P_{\infty}} \cap G_{P_0} = H(q)$, so $\mathfrak{u^{ab}} \in \set{h(\lambda_0)^{\pm 1}}$}
\li \If $\mathfrak{u^{ab}} = h(\lambda_0)$ 
\zi \Then
\li     Let $w$ be the $\SLP$ for $(\mathfrak{u^{ab} z}_2^{-1})^{\mathfrak{z}_1^{-1}} r^{-1}$ \label{main_alg_get_slp1_ree} 
\li     \Return $w$
\zi \Else
\li     Let $w$ be the $\SLP$ for $((\mathfrak{u^{ab}})^{-1} \mathfrak{z}_2^{-1})^{\mathfrak{z}_1^{-1}} r^{-1}$ \label{main_alg_get_slp2_ree} 
\li     \Return $w$
    \End 
\zi \kw{end}
\end{codebox}
\end{figure}

\subsection{Correctness and complexity}

\begin{thm} \label{thm_element_to_slp_ree}
Algorithm $\ref{alg:ree_element_to_slp}$ is correct, and is a Las Vegas algorithm.
\end{thm}
\begin{proof}
  First observe that since $r$ is randomly chosen, we obtain it
  as an $\SLP$. 

The elements $\mathfrak{z_1}$ and $\mathfrak{z_2}$ can be constructed using Lemma \ref{ree_row_operations}, so we can obtain them as $\SLP$s.

The element $\mathfrak{u}$ constructed at line \ref{main_alg_row_op3_ree} clearly has trace $x$.
The element $\Upsilon$ interchanges $P_{\infty}$ and $P_0$, hence $S(0, 1, 0)^{\Upsilon} \in G_{P_0}$ and $\mathfrak{u}$ can be computed using Lemma \ref{ree_row_operations}, so we
obtain it as an $\SLP$. From Proposition \ref{ree_conjugacy_classes}
we know that $\mathfrak{u}$ is conjugate to $h(\lambda_0)^{\pm 1}$, for some $\lambda_0 \in \F_q^{\times}$, and therefore fixes two points of $\OV$. Hence 
the elements found at lines \ref{main_alg_row_op4_ree}
and \ref{main_alg_row_op5_ree} can be computed using Lemma \ref{ree_row_operations}, so we obtain
them as $\SLP$s.

Finally, the elements that determine $w$ have
been constructed as $\SLP$s, and it is clear that if we evaluate $w$ we obtain
$g$. Hence the algorithm is Las Vegas and the theorem follows.
\end{proof}

\begin{thm} \label{thm_element_to_slp_complexity}
Algorithm $\ref{alg:ree_element_to_slp}$ has expected time complexity $\OR{\xi + \log(q)^3}$ field
operations and the length of the
returned $\SLP$ is $\OR{(\log(q) \log\log(q))^2}$.
\end{thm}
\begin{proof}
It follows immediately from Lemma \ref{ree_row_operations} that lines
\ref{main_alg_row_op1_ree}, \ref{main_alg_row_op2_ree},
\ref{main_alg_row_op3_ree}, \ref{main_alg_row_op4_ree} and
\ref{main_alg_row_op5_ree} use $\OR{\log(q)^3}$ field operations.

From Corollary \ref{cl_random_selections}, the expected time to find $r$ is $\OR{\xi}$ field operations. Half of the elements of $\F_q^{\times}$ are squares, and $x$ is uniformly random, hence the expected time of the outer repeat statement is $\OR{\xi + \log(q)^3}$ field operations.

Obtaining the fixed points of $u$, and performing the check at line
\ref{main_alg_find_point_ree} only amounts to considering eigenvectors,
hence uses $\OR{\log{q}}$ field operations. Thus the expected time complexity of
the algorithm is $\OR{\xi + \log(q)^3}$ field operations.

From Theorem \ref{thm_pre_step} each standard generator $\SLP$ has length $\OR{\log(q)(\log\log(q))^2}$ and hence $w$ has length $\OR{(\log(q)\log\log(q))^2}$ since Lemma \ref{ree_row_operations} increases the length by a factor $\log(q)$.
\end{proof}

\section{Conjugates of the standard copy}
\label{section:ree_conjugacy}

Assume that we are given a conjugate $G$ of $\Ree(q)$. We consider the problem of constructing $g \in\GL(7, q)$ such that $G^g =
\Ree(q)$, thus obtaining an algorithm that constructs effective
isomorphisms from any conjugate of $\Ree(q)$ to the standard copy.



\begin{thm} \label{thm_ree_conjugacy} 
Assume an oracle for the discrete logarithm
  problem in $\F_q$. There exists a Las Vegas algorithm that, given a
  conjugate $G = \gen{X}$ of $\Ree(q)$, constructs $g \in \GL(7, q)$ such that
  $\gen{X}^g = \Ree(q) = \mathfrak{S}$. The algorithm has expected time complexity $\OR{\xi \log\log(q) + \log(q)^2 + \chi_D}$ field operations.
\end{thm}
\begin{proof}
We prove the result by exhibiting the algorithm. Let $V = \gen{e_1, \dotsc, e_7}$ be the natural module for $\mathfrak{S}$.
\begin{enumerate}
\item Find a random involution $j_G \in G$. Let $j_\mathfrak{S} = h(-1) \in \mathfrak{S}$. By Corollary \ref{cl_random_selections} the expected time is $\OR{\xi + \log(q) \log\log(q)}$.

\item Compute generators for $\Cent_G(j_G)$ using \cite{bray00}, and generators for $C^G = \Cent_G(j_G)^{\prime}$ by taking commutators
  of the generators of $\Cent_G(j_G)$. Observe that $\Cent_\mathfrak{S}(j_\mathfrak{S}) = \gen{\Upsilon, h(\omega), S(0, 1, 0)}$ \cite{MR2653247} and similarly compute
  generators for $C^\mathfrak{S} = \Cent_\mathfrak{S}(j_\mathfrak{S})^{\prime}$. Similarly as in the proof of Corollary \ref{cl_find_stab_element}, the expected time is $\OR{1}$.
\item Use the MeatAxe to decompose the module of $C^G$ into
  its direct summands $V^G_3$ and $V^G_4$ of dimension $3$ and $4$. Decompose the module of $C^\mathfrak{S}$ into $V^\mathfrak{S}_3 = \gen{e_2, e_4, e_6}$ and $V^\mathfrak{S}_4 = \gen{e_1, e_3, e_5, e_7}$. Hence obtain change-of-bases $c_G$ and $c_\mathfrak{S}$ which exhibit the direct sums, with the $3$-dimensional
  submodules coming first. Let $C^G_3$ and $C^G_4$ be the projections
  of $C^G$ acting on the $3$-space and $4$-space, respectively, and
  similarly define $C^\mathfrak{S}_3$ and $C^\mathfrak{S}_4$. Since we have $\OR{1}$ generators for $C^G$ and $C^\mathfrak{S}$, the expected time is $\OR{\log(q)}$. Note that we also obtain a bijection between the generators of $C^G$ and $C^G_4$ or $C^G_3$, respectively, and similarly for $C^\mathfrak{S}$.
\item Use
  Theorem \ref{thm_psl_recognition} to constructively recognise
  $C^G_4$ and $C^\mathfrak{S}_4$ and obtain standard generators $Y^G_4$ and $Y^\mathfrak{S}_4$ for these groups
  as $\SLP$s in the input generators. Evaluate the $\SLP$s on the generators of $C^G$ and $C^\mathfrak{S}$ and use $c_G$ and $c_\mathfrak{S}$ to project the resulting matrices to the $3$-spaces. Hence also obtain standard generators $Y^G_3$ and $Y^\mathfrak{S}_3$ for $C^G_3$ and $C^\mathfrak{S}_3$. The expected time is $\OR{(\xi + \log(q)\log\log(q))\log\log(q) + \chi_D}$. Note that $\abs{Y^G_4} = \abs{Y^G_3} = \abs{Y^\mathfrak{S}_4} = \abs{Y^\mathfrak{S}_3}$.
\item By Proposition \ref{pr_inv_cent_module_structure}, $V^\mathfrak{S}_4 \cong \mathfrak{A}^{\varphi^{i_\mathfrak{S}}} \otimes \mathfrak{A}^{\varphi^{k_\mathfrak{S}}}$, for some $1 \leqslant i_\mathfrak{S} < k_\mathfrak{S} \leqslant 2m + 1$. Similarly, $V^G_4 \cong \mathfrak{A}^{\varphi^{i_G}} \otimes \mathfrak{A}^{\varphi^{k_G}}$. Use the MeatAxe together with $Y^G_4$ and $Y^\mathfrak{S}_4$ to obtain $1 \leqslant k \leqslant 2m+1$ such that $V^G_4 \cong (V^\mathfrak{S}_4)^{\varphi^k}$. Hence obtain a change-of-basis $c_4$ between these. Then $(C^G_4)^{c_4} = C^\mathfrak{S}_4$. The expected time is $\OR{\log(q)^2}$. 
\item Similarly, use the MeatAxe together with $Y^G_3$ and $Y^\mathfrak{S}_3$ to construct a change-of-basis $c_3$ from $V^G_3$ to $(V^\mathfrak{S}_3)^{\varphi^k}$. Then $(C^G_3)^{c_3} = C^\mathfrak{S}_3$. The expected time is $\OR{\log(q)}$.
\item Let $c_7$ be the diagonal join of $c_3$ and $c_4$. Let $c = c_g c_7 c_\mathfrak{S}^{-1}$. Then $(C^G)^c = C_\mathfrak{S}$.
\item Now $C_\mathfrak{S} \leqslant G^c \cap \mathfrak{S}$, so $G^c$ must preserve a form which is preserved by $C_\mathfrak{S}$. Use the MeatAxe to construct the form $K$ preserved by $G^c$. By Proposition \ref{pr_conj_form}, $K = \antidiag(b, a, b, -a, b, a, b)$ for some $a, b \in (\F^{\times})^2$, up to a scalar multiple. Let $x = \sqrt{a}$, $y = \sqrt{b}$ and $c_J = \diag(y, x, y, x, y, x, y)$. Then $G^{c c_J}$ preserves the form $J$ and $c_J$ normalises $C_\mathfrak{S}$. The expected time is $\OR{1}$.
\item Now $G^{c c_J} < \Omega(7, q)$ and $C_\mathfrak{S} < G^{c c_J}$. By Proposition \ref{pr_inv_cent_normaliser}, $C_\mathfrak{S}$ is contained in at most four $\Omega(7, q)$-conjugates of $\mathfrak{S}$, so $G^{c c_J} = \mathfrak{S}$ with probability at least $1/4$. Use Theorem \ref{thm_ree_standard_recognition} to test this. The expected time is $\OR{\sigma_0(\log(q))\log(q)}$.
\end{enumerate}
If any of the tests or Las Vegas algorithms used fail, we start again from the beginning. In total, the expected time complexity is $\OR{\xi \log\log(q) + \log(q)(\sigma_0(\log(q)) + \log(q)) + \chi_D}$ field operations. This proves the result.
\end{proof}

\subsection{Main theorem}

\begin{proof}[Proof of Theorem \ref{main_theorem}]
  The algorithm providing $\Psi$ follows from Theorem
  \ref{thm_ree_conjugacy}. Since $\Psi$ and $\Psi^{-1}$ are just conjugations, they can be computed using $\OR{1}$ field operations, so they are effective.

Constructive membership testing in
  $\Ree(q)$ follows from Theorem \ref{thm_pre_step} and Algorithm
  \ref{alg:ree_element_to_slp}. For constructive membership testing in
  $\gen{X}$ we first map the element to $\Ree(q)$ using $\Psi$, then express it as an $\SLP$.
\end{proof}

\section{Implementation and performance}

Implementations of the algorithms are available in $\MAGMA$. The
implementations use the existing $\MAGMA$ implementations of the
algorithms described in \cite{bray00}, \cite{lg95},
\cite{crlg95}, \cite{psl_recognition},
\cite{smallerfield} and \cite{meataxe, better_meataxe}.

We have benchmarked the computation of generating sets for
stabilisers, in other words most of the algorithm from Theorem
\ref{thm_pre_step}. This is shown in Figure
\ref{fig:stab_benchmark}. For each field size $q = 3^{2m + 1}$,
generating sets for stabilisers of $100$ random points were
computed, and the average running time for each call is listed. The
amount of this time that was spent in discrete logarithm computations
outside \cite{psl_recognition}, $\SLP$ evaluations and in
\cite{psl_recognition} is also indicated. Note that the algorithm of
\cite{psl_recognition} also uses a discrete logarithm oracle.

When $2m + 1$ has a \lq\lq small'' prime divisor, finite field
arithmetic in $\F_q$ in $\MAGMA$ is particularly fast. This is because
$\MAGMA$ uses Zech logarithms for finite fields up to a certain size,
and for larger fields it tries to find a subfield smaller than this
size. If this is possible the arithmetic in the larger field will be very fast. To avoid jumps in the figure, and to properly measure field
operations, we have turned off this optimisation, and have in each case divided by the time required for $10^6$ multiplications of random pairs of field elements.


\begin{figure}[ht]
\includegraphics[scale=0.75]{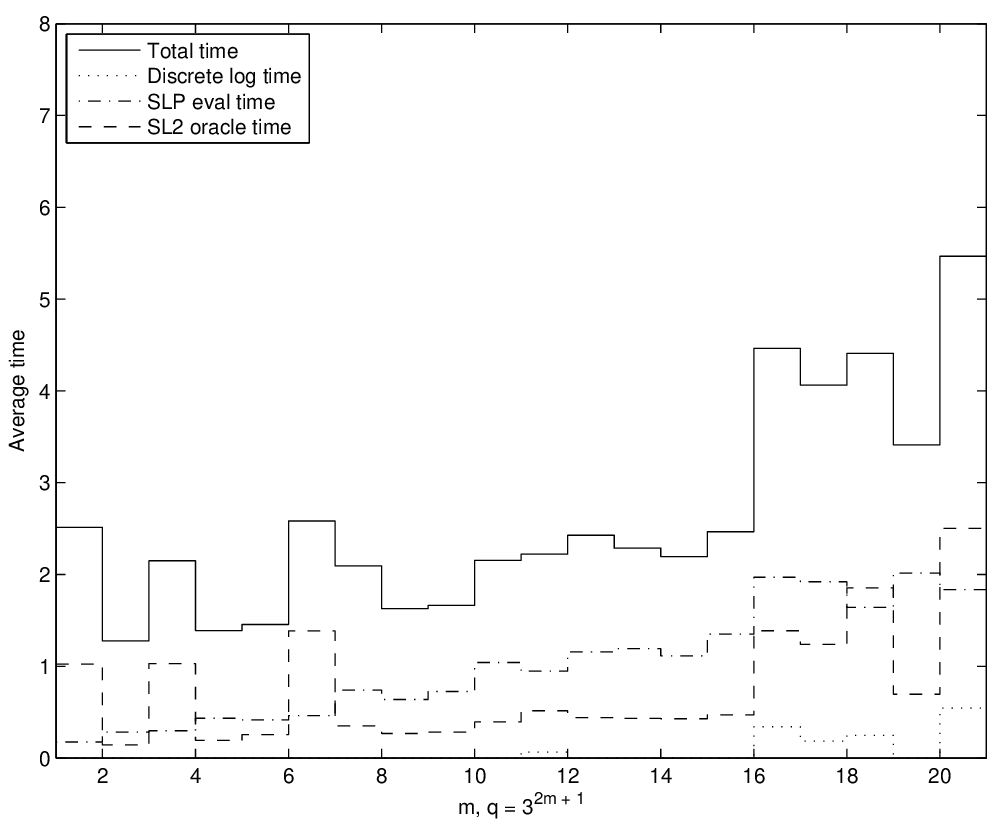}
\caption{Benchmark of stabiliser computation}
\label{fig:stab_benchmark}
\end{figure}

In the same fashion, we have benchmarked the conjugation algorithm from Theorem \ref{thm_ree_conjugacy}. This is shown in Figure \ref{fig:conj_benchmark}.

\begin{figure}[ht]
\includegraphics[scale=0.75]{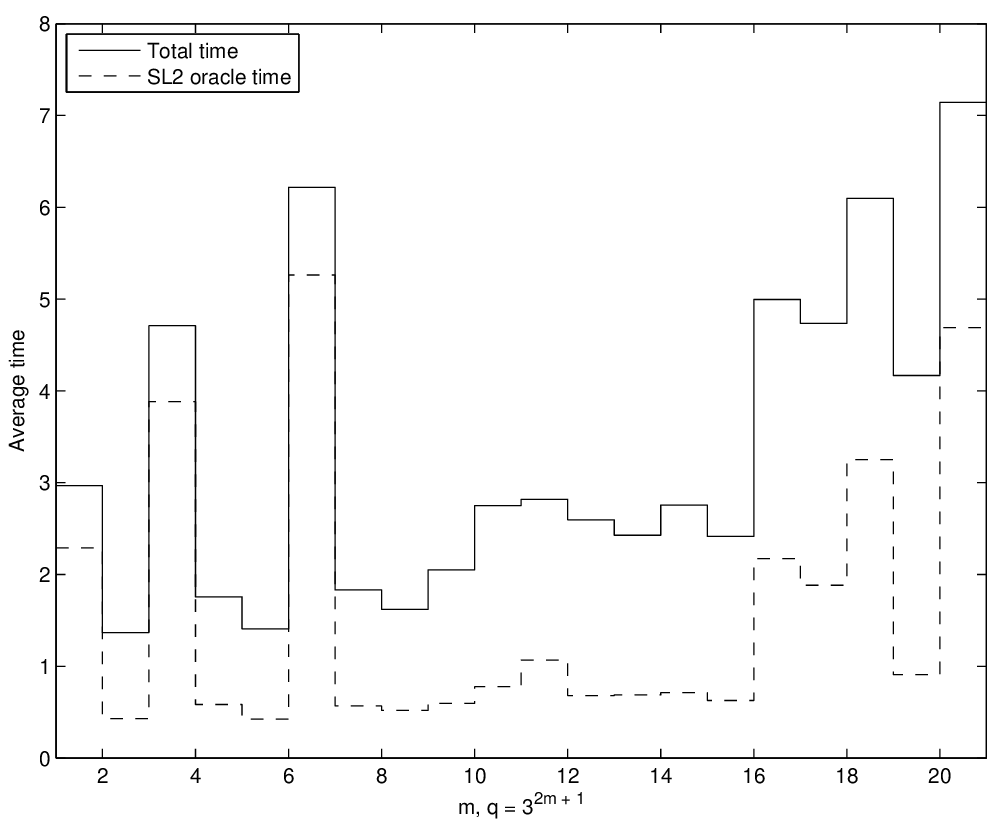}
\caption{Benchmark of Ree conjugation}
\label{fig:conj_benchmark}
\end{figure}


All benchmarks were carried out using $\MAGMA$ V2.18-2, Intel64
flavour, on a PC with an Intel Core2 CPU running at $2$ GHz, and with
$2$ GB of RAM. The largest value of $m$ in the tests was $20$, since discrete logarithm computations became very slow in $\F_{3^{43}}$. 







\bibliographystyle{amsplain}
\bibliography{ree}

\end{document}